\documentclass[12pt,oneside,english]{amsart}
\usepackage[T1]{fontenc}
\usepackage[latin9]{inputenc}
\usepackage{mathrsfs}
\usepackage{amstext}
\usepackage{amsthm}
\usepackage{amssymb}

\makeatletter
\numberwithin{equation}{section}
\numberwithin{figure}{section}
\theoremstyle{plain}
\newtheorem*{thm*}{\protect\theoremname}
\theoremstyle{plain}
\newtheorem{thm}{\protect\theoremname}
\theoremstyle{plain}
\newtheorem{lem}[thm]{\protect\lemmaname}
\theoremstyle{plain}
\newtheorem{cor}[thm]{\protect\corollaryname}

\usepackage{fullpage}

\makeatother

\usepackage{babel}
\providecommand{\corollaryname}{Corollary}
\providecommand{\lemmaname}{Lemma}
\providecommand{\theoremname}{Theorem}

\usepackage{amsrefs}

\begin{document}

\title{Cocycle Conjugacy of Free Bogoljubov Actions of $\mathbb{R}$}

\author{Joshua Keneda and Dimitri Shlyakhtenko}

\thanks{Research supported by NSF grant DMS-1762360.}

\email{jkeneda@southplainscollege.edu, shlyakht@math.ucla.edu}

\address{Department of Mathematics, UCLA, Los Angeles, CA 90095}
\address{Mathematics and Engineering Department, South Plains College, Levelland, TX 79336}

\begin{abstract}
We show that Bogoljubov actions of $\mathbb{R}$ on the free group
factor $L(\mathbb{F}_{\infty})$ associated to sums of
infinite multiplicity trivial and certain mixing representations are cocycle conjugate
if and only if the underlying representations are conjugate.
\end{abstract}

\maketitle

\section{Introduction}

Recall that two actions $\beta_{t}$, $\gamma_{t}$ of $\mathbb{R}$
on a von Neumann algebra $M$ are said to be \emph{conjugate} if $\beta_{t}=\alpha\circ\gamma_{t}\circ\alpha^{-1}$
for some automorphism $\alpha$ of $M$. The actions are said to be \emph{cocycle conjugate} if there exists a strongly continuous one-parameter
family of unitaries $u_{t}\in M$ and an automorphism $\alpha$ of
$M$ so that
\[
\beta_{t}(x)=\alpha(\textrm{Ad}_{u_{t}}(\gamma_{t}(\alpha^{-1}(x)))),\qquad\forall x\in M;
\]
in other words, $\beta_{t}$ and $\textrm{Ad}_{u_{t}}\circ\gamma_{t}$
are conjugate. Cocycle conjugacy is clearly a weaker notion of equivalence
than conjugacy. 

A consequence of cocycle conjugacy is an isomorphism between the crossed
product von Neumann algebras $M\rtimes_{\beta}\mathbb{R}$ and $M\rtimes_{\gamma}\mathbb{R}$.
This isomorphism takes $M$ to $M$ and sends the unitary $U_{t}\in L(\mathbb{R})\subset M\rtimes_{\beta}\mathbb{R}$
implementing the automorphism $\beta_{t}$ to the unitary $u_{t}V_{t}$,
where $V_{t}\in M\rtimes_{\gamma}\mathbb{R}$ is the implementing unitary
for $\gamma_{t}$. 

An important class of automorphisms of the free group factor $L(\mathbb{F}_{\infty}$)
are so-called \emph{free Bogoljubov automorphisms}, which are defined
using Voiculescu's free Gaussian functor. As a starting point, we
write $L(\mathbb{F}_{\infty})=W^{*}(S_{1},S_{2},\dots)$ where $S_{j}$
are an infinite free semicircular system. The closure in the $L^{2}$
norm of their real-linear span is an infinite dimensional real Hilbert
space. Voiculescu proved that any automorphism of that Hilbert space
extends to an automorphism of $L(\mathbb{F}_{\infty})$. In particular,
any representation of $\mathbb{R}$ on an infinite dimensional Hilbert
space canonically gives an action of $\mathbb{R}$ on $L(\mathbb{F}_{\infty})$. 

Motivated by the approach in \cite{Classification}, we prove the following theorem,
which states that for a large class of Bogoljubov automorphisms, cocycle
conjugacy and conjugacy are equivalent to conjugacy of the underlying
representations and thus gives a classification of such automorphisms
up to cocycle conjugacy.
\begin{thm*}
\label{thm:SameSpectrum-1}Let $\pi_{1},\pi_{1}'$ be two mixing
orthogonal representations of $\mathbb{R}$, and assume that $\pi_{1}\otimes\pi_{1}\cong\pi_{1}$,
$\pi_{1}'\otimes\pi_{1}'\cong\pi_{1}'$. Denote by $\mathbf{1}$ the
trivial representation of $\mathbb{R}$. Let
\[
\pi=(\mathbf{1}\oplus\pi_{1})^{\oplus\infty},\qquad\pi'=(\mathbf{1}\oplus\pi_{1}')^{\oplus\infty},
\]
and let $\alpha$ (resp.,
$\alpha'$) be the associated free Bogoljubov actions of $\mathbb{R}$ on $L(\mathbb{F}_{\infty})$.
Then $\alpha$ and $\alpha'$ are cocycle conjugate iff the representations
$\pi^{\oplus\infty}$ and $(\pi')^{\oplus\infty}$ are conjugate,
i.e.
\[
\pi^{\oplus\infty}=V(\pi')^{\oplus\infty}V^{-1}
\]
for some orthogonal isomorphism $V$ of the underlying real Hilbert
spaces.
\end{thm*}
It is worth noting that the conjugacy class of a representation of
$\mathbb{R}$ on a real Hilbert space is determined by the measure
class of its spectral measure (a measure on $\mathbb{R}$ satisfying
$\mu(-X)=\mu(X)$ for all Borel sets $X$) and a multiplicity function
which is measurable with respect to that class (for the purposes of
our Theorem, we may assume that this multiplicity function is identically
infinite). Our Theorem then states that, for Bogoljubov actions
satisfying the hypothesis of the Theorem, cocycle conjugacy occurs
if and only if these measure classes are the same.

\section{Preliminaries on conjugacy of automorphisms}

\subsection{Crossed products.}

If $M$ is a type II$_{1}$ factor and $\alpha_{t}:\mathbb{R}\to\textrm{Aut}(M)$
is a one-parameter group of automorphisms, the \emph{crossed product}
$M\rtimes_{\alpha}\mathbb{R}$ is of type II$_{\infty}$ with a canonical
trace $Tr$. Furthermore, the crossed product construction produces
in a canonical way a distinguished copy of the group algebra $L(\mathbb{R})$
inside the crossed product algebra. We denote this copy by $L_{\alpha}(\mathbb{R})$.
The relative commutant $L_{\alpha}(\mathbb{R})'\cap M\rtimes_{\alpha}\mathbb{R}$
is generated by $L_{\alpha}(\mathbb{R})$ and the fixed point algebra
$M^{\alpha}$. 

\subsection{Conjugacy of actions.}

Recall that if $\beta_{t}$ and $\gamma_{t}$ are cocycle conjugate,
each choice of a cocycle conjugacy produces an isomorphism $\Pi_{\gamma,\beta}$
of the crossed product algebras $M\rtimes_{\beta}\mathbb{R}$ and
$M\rtimes_{\gamma}\mathbb{R}$. Note that $\Pi_{\gamma,\beta}$ does
\emph{not} necessarily map $L_{\beta}(\mathbb{R})$ to $L_{\gamma}(\mathbb{R})$.
In fact, as we shall see, this is rarely the case even if we compare
the image $\Pi_{\gamma,\beta}(L_{\beta}(\mathbb{R}))$ with $L_{\gamma}(\mathbb{R})$
up to a weaker notion of equivalence, $\prec_{M\rtimes_{\gamma}\mathbb{R}}$
which was introduced by Popa in the framework of his deformation-rigidity
theory. Indeed, in parallel to Theorem 3.1 in \cite{Classification},
we show that, very roughly, conjugacy of $\Pi_{\gamma,\beta}(L_{\beta}(\mathbb{R}))$
and $L_{\gamma}(\mathbb{R})$ inside the crossed product is essentially
equivalent (up to compressing by projections) to \emph{conjugacy}
of the actual actions by an inner automorphism of $M$. 
\begin{thm}
\label{embedding} Let $M$ be a tracial von Neumann algebra with
a fixed faithful normal trace $\tau$. Suppose $\alpha,\beta:\mathbb{R}\rightarrow\text{Aut}(M)$
are two trace-preserving actions of $\mathbb{R}$ on $M$ which are
cocycle conjugate, and suppose that the only finite-dimensional $\alpha$-invariant
subspaces of $L^{2}(M)$ are those on which $\alpha$ acts trivially.
Fix any $q\in M^{\beta}$ a nonzero projection. The following are
equivalent:

(a) There exists a nonzero projection $r\in L_{\beta}(\mathbb{R})$
such that 
\[
\Pi_{\alpha,\beta}(L_{\beta}(\mathbb{R})qr)\prec_{M\rtimes_{\alpha}\mathbb{R}}L_{\alpha}(\mathbb{R})
\]

(b) There exists a nonzero partial isometry $v\in M$ such that $v^{*}v\in qM^{\beta}q$,
$vv^{*}\in M^{\alpha}$, and for all $x\in M$,
\[
\alpha_{t}(vxv^{*})=v\beta_{t}(x)v^{*}.
\]
\end{thm}

\begin{proof}
To see that (a) implies (b), take $r$ as in (a), so that $\Pi_{\alpha,\beta}(L_{\beta}(\mathbb{R})qr)\prec_{M\rtimes_{\alpha}\mathbb{R}}L_{\alpha}(\mathbb{R})$,
and take $w_{t}\in M$ with $\text{Ad }w_{t}\circ\alpha_{t}=\beta_{t}$.

First, we claim that there's a $\delta>0$ for which there exist $x_{1},...,x_{k}\in qM$
with 
\[
\sum_{i,j=1}^{k}|\tau(x_{i}^{*}w_{t}\alpha_{t}(x_{j}))|^{2}\ge\delta
\]
for all $t$. Suppose for a contradiction that no such $\delta$ exists.
Then we can find a net $(t_{i})_{i\in I}$ such that 
\[
\lim_{i}\tau(x^{*}w_{t_{i}}\alpha_{t_{i}}(y))=0
\]
for any $x,y\in qM$.

But then for any $p,p'$ finite trace projections in $L_{\alpha}(\mathbb{R})$,
$s,s'\in\mathbb{R}$, and $x,y\in M$, we have (in the 2-norm from
the trace on $M\rtimes_{\alpha}\mathbb{R}$): 
\begin{align*}
\|E_{L_{\alpha}(\mathbb{R})}(p\lambda_{\alpha}(s)^{*}x^{*}\Pi_{\alpha,\beta}(\lambda_{\beta}(t_{i})q)y\lambda_{\alpha}(s')p')\|_{2} & =\|\lambda_{\alpha}(s)^{*}pE_{L_{\alpha}(\mathbb{R})}(x^{*}q\Pi_{\alpha,\beta}(\lambda_{\beta}(t_{i})qy)p'\lambda_{\alpha}(s'))\|_{2}\\
 & =\|pE_{L_{\alpha}(\mathbb{R})}((qx)^{*}w_{t_{i}}\alpha_{t_{i}}(qy))p'\lambda_\alpha(s'+t_{i})\|_{2}\\
 & =\|E_{L_{\alpha}(\mathbb{R})}((qx)^{*}w_{t_{i}}\alpha_{t_{i}}(qy))pp'\|_{2}\rightarrow0,
\end{align*}
where the last equality follows from the fact that $(qx)^{*}w_{t_{i}}\alpha_{t_{i}}(qy)\in M$,
so 
\[
E_{L_{\alpha}(\mathbb{R})}((qx)^{*}w_{t_{i}}\alpha_{t_{i}}(qy))=\tau((qx)^{*}w_{t_{i}}\alpha_{t_{i}}(qy)),
\]
and the latter term goes to zero by supposition for any $x,y\in M$.

Now note that linear combinations of terms of the form $x\lambda_{\alpha}(s)p$
(resp. $y\lambda_{\beta}(s')p'$) as above are dense in $L^{2}(M\rtimes_{\alpha}\mathbb{R},Tr),$
so by approximating $\Pi_{\alpha,\beta}(r)a$, (resp. $\Pi_{\alpha,\beta}(r)b$)
with such sums for any $a,b\in M\rtimes_{\alpha}\mathbb{R}$, it follows
from the above estimate that 
\[
\|E_{L_{\alpha}(\mathbb{R})}(a^{*}\Pi_{\alpha,\beta}(\lambda_{\beta}(t_{i})qr)b)\|_{2}\rightarrow0.
\]
But this contradicts $\Pi_{\alpha,\beta}(L_{\beta}(\mathbb{R})qr)\prec_{M\rtimes_{\alpha}\mathbb{R}}L_{\alpha}(\mathbb{R})$,
so the $\delta>0$ of our above claim exists.

We can thus find $\delta>0$, $x_{1},...,x_{k}\in qM$ such that $\sum_{i,j=1}^{k}|\tau(x_{i}^{*}w_{t}\alpha_{t}(x_{j}))|^{2}\ge\delta$
for all $t$.

Let us now consider the space $B(L^{2}(M))$ of bounded operators on
$L^{2}(M,\tau)$ and the rank-one orthogonal projection $e_{\tau}$
onto $1\in L^{2}(M,\tau)$. We can identify $(B(L^{2}(M)),e_{\tau})$
with the basic construction $\langle M, e_\tau \rangle$ for $\mathbb{C}\subset M$, so that $e_{\tau}$
is the Jones projection. Let $\hat{\tau}$ be the usual trace on $B(L^{2}(M))$,
satisfying $\hat{\tau}(xe_{\tau}y)=\tau(xy)$ for all $x,y\in M$.
Finally, for a finite-rank operator $Q=\sum_{i}y_{i}e_{\tau}z_{i}^{*}$
with $y_{i},z_{i}\in M$, let
\[
T_{M}(Q)=\sum y_{i}z_{i}^{*}\in M.
\]
Then $T_{M}$ extends to a normal operator-valued weight from the
basic construction to $M$ satisfying $\hat{\tau}=\tau\circ T_{M}$
(i.e. $T_{M}$ is the pull-down map). 

Consider now the positive element 
\[
X=\sum_{i=1}^{k}x_{i}e_{\tau}x_{i}^{*},
\]
together with the following normal positive linear functional on $\langle M,e_{\tau}\rangle$:
\[
\psi(T)=\sum_{i=1}^{k}\hat{\tau}(e_{\tau}x_{i}^{*}Tx_{i}e_{\tau}).
\]

Note that $T_{M}(X)=\sum_{i=1}^{k}x_{i}x_{i}^{*}\in M$, so in particular
$\|T_{M}(X)\|<\infty$.

For every $t\in\mathbb{R}$, we have: 
\begin{align*}
\psi(\beta_{t}(X)) & =\sum_{i,j}\hat{\tau}(e_{\tau}x_{i}^{*}w_{t}\alpha_{t}(x_{j})e_{\tau}\alpha_{t}(x_{j})^{*}w_{t}^{*}x_{i}e_{\tau})\\
 & =\sum_{i,j}|\tau(x_{i}^{*}w_{t}\alpha_{t}(x_{j})|^{2}\ge\delta>0.
\end{align*}

Now consider $K$, the ultraweak closure of the convex hull of $\{\beta_{t}(X):t\in\mathbb{R}\}$
inside $q\langle M,e_{\tau}\rangle q$. Note that by normality of
$\psi$, $\psi(x)\ge\delta$ for any $x\in K$.

Since $K$ is convex and $\|\cdot\|_{2}$-closed, there exists a unique
$X_{0}\in K$ of minimal 2-norm. But since the 2-norm is invariant
under $\beta$, we must have that $\|\beta_{t}(X_{0})\|_{2}=\|X_{0}\|_{2}$
for all $t$, so by uniqueness of the minimizer, $X_{0}$ is itself
fixed by the extended $\beta$ action (and nonzero since $\psi(X_{0})\ge\delta$).
Also, by ultraweak lower semicontinuity of $T_{M}$, we know that
$\|T_{M}(X_{0})\|\le\|T_{M}(X)\|<\infty$.

Take a nonzero spectral projection $e$ of $X_{0}$. Then $e$ is
still $\beta$-invariant and satisfies $\|T_{M}(e)\|<\infty$. But
this means that $\hat{\tau}(e)=\tau(T_{M}(e))<\infty$, so $e$ must
be a finite rank projection, since $\hat{\tau}$ corresponds to the
usual trace $Tr$ on the trace-class operators in $B(L^{2}(M),\tau)$.

Now since $e_{\tau}$ has central support 1 in $\langle M,e_{\tau}\rangle$
(and because $e_{\tau}$ is minimal), we have that there exists $V$
a partial isometry in $\langle M,e_{\tau}\rangle$ such that $V^{*}V=f\le e$
and $VV^{*}=e_{\tau}$. We remark that $f$ remains $\beta$-invariant,
since $e$ was finite rank, and our finite-dimensional invariant subspaces
are all fixed by the action. Note also that $e\le q$ (since $X_{0}\in q\langle M,e_{\tau}\rangle q$),
so that $V=Vq=e_{\tau}V$.

Applying the pull-down lemma, we see that: 
\[
V=e_{\tau}V=e_{\tau}(T_{M}(e_{\tau}V))=e_{\tau}T_{M}(V).
\]
Set $v=T_{M}(V)$ and note that because $\|T_{M}(V^{*}V)\| \le \|T_{M}(e)\|<\infty$,
we have $v\in M$, and $V=e_{\tau}v$.

Since $e_{\tau}\langle M,e_{\tau}\rangle e_{\tau}=\mathbb{C}e_{\tau}$,
and since $V$ is left-supported by $e_{\tau}$, we have that for
each $t$ there exists a $\lambda_{t}\in\mathbb{C}$ such that $\lambda_{t}e_{\tau}=Vw_{t}{\alpha_{t}}(V^{*})$.
Note that since $Vw_{t}{\alpha_{t}}(V^{*})(Vw_{t}{\alpha_{t}}(V^{*}))^{*}=Vw_{t}{\alpha_{t}}(V^{*}V)w_{t}^{*}V^{*}=V{\beta_{t}}(e)V^{*}=VV^{*}=e_{\tau}$,
the last equality of the previous sentence implies that $\lambda_{t}\overline{\lambda_{t}}=1$.
We also have: 
\begin{align*}
e_{\tau}\lambda_{t}\alpha_{t}(v) & =\lambda_{t}e_{\tau}\alpha_t(v)=\lambda_{t}e_{\tau}{\alpha_t}(V)\\
 & =Vw_{t}{\alpha_{t}}(V^{*}V)=V{\beta_{t}}(e)w_{t}=Vw_{t}\\
 & =e_{\tau}vw_{t}.
\end{align*}

Thus, applying the pull-down map, we have that $\lambda_{t}\alpha_{t}(v)=vw_{t}$,
and, replacing $v$ by its polar part if necessary, we've found a
partial isometry in $M$, conjugation by which intertwines the actions.
We have for any $x\in M$: 
\begin{align*}
\alpha_{t}(vxv^{*})=\alpha_{t}(v)\alpha_{t}(x)\alpha_{t}(v^{*})=\overline{\lambda_{t}}vw_{t}\alpha_{t}(x)w_{t}^{*}v^{*}\lambda_{t}=v\beta_{t}(x)v^{*}.
\end{align*}

Furthermore, with some applications of $\alpha_{t}(v)=\overline{\lambda_{t}}vw_{t}$,
we see that 
\[
\beta_{t}(v^{*}v)=w_{t}\alpha_{t}(v^{*}v)w_{t}^{*}=w_{t}(w_{t}^{*}v^{*}\lambda_{t})(\overline{\lambda_{t}}vw_{t})w_{t}^{*}=v^{*}v,
\]
and 
\[
\alpha_{t}(vv^{*})=(\overline{\lambda_{t}}vw_{t})(w_{t}^{*}v^{*}\lambda_{t})=vv^{*},
\]
so we've found the promised intertwiner.

Conversely, assume that we have $v\in M$ satisfying $v^{*}v\in qM^{\beta}q$,
$vv^{*}\in M^{\alpha}$, and $\alpha_{t}(vxv^{*})=v\beta_{t}(x)v^{*}$
for all $x\in M$. Take $w_{t}\in M$ with $\text{Ad }w_{t}\circ\alpha_{t}=\beta_{t}$.
Then, as above, we have $vw_{t}=\lambda_{t}\alpha_{t}(v)$, for some
$\lambda_{t}\in \mathbb{T}$. Multiplying both sides by $\overline{\lambda_{t}}$
and absorbing this factor into $w_{t}$, we may assume without loss
of generality that $\lambda_{t}=1$ for all $t$, so we have $vw_{t}=\alpha_{t}(v)$.

Now let $\lambda_{t}^{\alpha}$ (resp., $\lambda_{t}^{\beta}$) denote
the canonical unitaries that implement the respective actions on $M$
in the crossed product $M\rtimes_{\alpha}\mathbb{R}$ (resp., $M\rtimes_{\beta}\mathbb{R}$).
Then the relation $vw_{t}=\alpha_{t}(v)$ implies $v\Pi_{\alpha,\beta}(\lambda_{t}^{\beta})=\lambda_{t}^{\alpha}v$.
Furthermore, for any finite trace projection $r\in L_{\beta}(\mathbb{R})$,
we have $v\Pi_{\alpha,\beta}(qr)=vq\Pi_{\alpha,\beta}(r)=v\Pi_{\alpha,\beta}(r)\neq0$,
so $v^{*}$ is a partial isometry that witnesses $\Pi_{\alpha,\beta}(L_{\beta}(\mathbb{R})qr)\prec_{M\rtimes_{\alpha}\mathbb{R}}L_{\alpha}(\mathbb{R})$
(e.g. see condition (4) of Theorem F.12 in \cite{BO}). Thus, (b) implies (a). 
\end{proof}

\section{Cocycle conjugacy of Bogoljubov Automorphisms}

\subsection{Free Bogoljubov automorphisms.}

Let $\pi$ be an orthogonal representation of $\mathbb{R}$ on a real
Hilbert space $H_{\mathbb{R}}$. Recall that Voiculescu's free Gaussian
functor associates to $H_{\mathbb{R}}$ a von Neumann algebra
\[
\Phi(H_{\mathbb{R}})=\{s(h):h\in H_{\mathbb{R}}\}''\cong L(\mathbb{F}_{\dim H_{\mathbb{R}}})
\]
where $s(h)=\ell(h)+\ell(h)^{*}$ and $\ell(h)$ is the creation operator
$\xi\mapsto h\otimes\xi$ acting on the full Fock space $\mathscr{F}(H)=\bigoplus_{n\geq0}(H_{\mathbb{R}}\otimes_{\mathbb{R}}\mathbb{C})^{\otimes n}$.
Denoting by $\Omega$ the unit basis vector of $(H_{\mathbb{R}}\otimes_{\mathbb{R}}\mathbb{C})^{\otimes0}=\mathbb{C}$,
it is well-known that the vector-state
\[
\tau(\cdot)=\langle\Omega,\cdot\Omega\rangle
\]
defines a faithful trace-state on $\Phi(H_{\mathbb{R}})$. Furthermore,
$\mathbb{R}$ acts on $\mathscr{F}(H)$ by unitary transformations
$U_{t}=\bigoplus_{n\geq0}(\pi\otimes1)^{\otimes n}$, and conjugation
by $U_{t}$ leaves $\Phi(H_{\mathbb{R}})$ globally invariant thus
defining a strongly continuous one-parameter family of \emph{free
Bogoljubov automorphisms}
\[
t\mapsto\alpha_{t}\in\textrm{Aut}(\Phi(H_{\mathbb{R}})).
\]

Note that if $\pi$ is such that $\pi\otimes\pi$ and $\pi$ are conjugate (as 
representations of $\mathbb{R}$), then the representation
$U_{t}$ is conjugate to $\mathbf{1}\oplus\pi^{\oplus\infty}$. 

A complete invariant for the orthogonal representation $\pi$ consists
of the absolute continuity class $\mathscr{C}_{\pi}$ of a measure
$\mu$ on $\mathbb{R}$ satisfying the symmetry condition $\mu(B)=\mu(-B)$
for all $\mu$-measurable sets $B$ and a $\mu$-measurable multiplicity
function $n:\mathbb{R}\to\mathbb{N}\cup\{+\infty\}$ satisfying $n(x)=n(-x)$
almost surely in $x$. In particular, assuming that $\pi\cong\pi\otimes\pi$
(i.e., that for some (hence any) probability measure $\mu\in\mathscr{C}_{\pi}$
that generates $\mathscr{C}_{\pi}$, $\mu*\mu$ and $\mu$ are mutually absolutely continuous), then the measure class $\mathscr{C}_{\pi}$
is an invariant of $\alpha$ up to conjugacy (since it can be recovered
from $U_{t}$, the unitary representation induced by $\alpha_{t}$
on $L^{2}(\Phi(H_{\mathbb{R}})$). 

Recall that the representation $\pi$ is said to be \emph{mixing}
if for all $\xi,\eta\in H_{\mathbb{R}}$, $\lim_{|t|\to\infty}\langle\xi,\pi(t)\eta\rangle\to0$.
This is equivalent to saying that for some (hence any) probability
measure $\mu\in\mathscr{C}_{\pi}$ that generates $\mathscr{C}_{\pi}$,
the Fourier transform satisfies $\hat{\mu}(t)\to0$ whenever $t\to\pm\infty$.

\subsection{Operator-valued semicircular systems.\label{subsec:OpValuSemSys}}

The crossed product $\Phi(H_{\mathbb{R}})\rtimes_{\alpha}\mathbb{R}$
has a description in terms of so-called operator-valued semicircular
systems (see \cite[Examples 2.8, 5.2]{A-valued}). Decompose $\pi=\bigoplus_{i\in I}\pi_{i}$
into cyclic representations $\pi_{i}$ with cyclic vectors $\xi_{i}$.
Let $\mu_{i}$ be the measure with Fourier transform $t\mapsto\langle\xi_{i},\pi_{i}(t)\xi_{i}\rangle$,
and denote by $\eta_{i}:L^{\infty}(\mathbb{R})\to L^{\infty}(\mathbb{R})$
the completely positive map given by
\[
\eta_{i}(f)(x)=\int f(y)d\mu(x-y).
\]
Then \cite[Proposition 2.18]{A-valued} shows that $\Phi(H_{\mathbb{R}})\rtimes_{\alpha}\mathbb{R}=W^{*}(L(\mathbb{R}),S_{i}:i\in I)$
where $S_{i}$ are free with amalgamation over $L(\mathbb{R})\cong L^{\infty}(\mathbb{R})$
and each $S_{i}$ is an $L^{\infty}(\mathbb{R})$-valued semicircular
system with covariance $\eta_{i}$. 

Operator-valued semicircular variables can be associated to any normal
self-adjoint completely positive map on $L(\mathbb{R})\cong L^{\infty}(\mathbb{R})$.
In particular, given any measure $K$ on $\mathbb{R}^{2}$ satisfying
$\pi_{x}K,\pi_{y}K\prec\textrm{Lebesgue measure}$ and $dK(x,y)=dK(y,x)$
(here $\pi_{x},\pi_{y}$ are projections on the two coordinate axes),
we can construct an $L^{\infty}(\mathbb{R})$-valued semicircular
variable $S=S_{K}$ with covariance
\[
\eta(f)(x)=\int f(y)dK(x,y).
\]
If $K'$ is absolutely continuous with respect to $K$, then $W^{*}(L^{\infty}(\mathbb{R}),S_{K'})\subset W^{*}(L^{\infty}(\mathbb{R}),S_{K})$;
if $K=\sum K_{j}$ with $K_{j}$ disjoint, then $S_{K_{j}}$ are free
with amalgamation over $L^{\infty}(\mathbb{R}).$ 

The algebra $W^{*}(L^{\infty}(\mathbb{R}),S_{K})$ is denoted $\Phi(L^{\infty}(\mathbb{R}),\eta)$.
For our choices of $\eta$ it is semifinite and $L^{\infty}(\mathbb{R})$
is in the range of a conditional expectation.

If $I\subset\mathbb{R}$ is a finite interval and $K$ is a measure
on $I^{2}$ satisfying $\pi_{x}K,\pi_{y}K\prec\textrm{Lebesgue measure}$
and $dK(x,y)=dK(y,x)$, then one can in a similar way associate a
completely positive map to $K$ and consider an $L^{\infty}(I)$-semicircular
variable $S_{K}$. This time, composition of the conditional expectation
onto $L^{\infty}(I)$ with integration with respect to (rescaled)
Lebesgue measure on $L^{\infty}(I)$ gives rise to a normal faithful
trace on the algebra $\Phi(L^{\infty}(I),\eta)=W^{*}(L^{\infty}(I),S_{K})$.

We call measures $K$ (on $\mathbb{R}^{2}$ or on the square of some
finite interval $I$) satisfying the conditions $\pi_{x}K,\pi_{y}K\prec\textrm{Lebesgue measure}$
and $dK(x,y)=dK(y,x)$ \emph{kernel measures.}

\subsection{Solidity of certain algebras generated by operator-valued semicircular
systems.}

Let $\eta$ be a normal self-adjoint completely positive map defined
on the von Neumann algebra $A = L^{\infty}(\mathbb{T})$ (with Haar measure on $\mathbb{T}$). By \cite[Corollary 4.2]{Thin}, if the $A,A$-bimodule associated to $\eta$
is mixing (see Def. 2.2 of that paper for a definition), then
$\Phi(A,\eta)$ is strongly solid, and in particular, solid: the relative
commutant of any diffuse abelian von Neumann subalgebra of $\Phi(A,\eta)$
is amenable. As noted in \cite[Proposition 7.3.4]{Thin} and its surrounding remarks, if $\mu$ is a measure on $\mathbb{T}$
so that its Fourier transform $\hat{\mu}$ satisfies $\lim_{n\to\pm\infty}\hat{\mu}(n)=0$
(i.e., $\mu$ is a measure associated to a mixing representation),
then the bimodule associated to the completely positive map $\eta:f\mapsto f*\mu$
is mixing. We record the following lemma, whose proof is straightforward
from the definition of mixing bimodules:
\begin{lem}
Suppose that $H,H'$ are mixing $A,A$-bimodules, $p_{0}\in A$
is a nonzero projection, and $K\subset H$ is an $A,A$-submodule.
Then:

(i) $H\oplus H'$ is mixing; (ii) $K$ is mixing; (iii) $p_{0}Hp_{0}$ is mixing as a $p_{0}A,p_{0}A$-bimodule.
\end{lem}

We now make use of this lemma.
\begin{lem}
\label{lem:ACTrick}Let $(K_{j}:j\in J)$ be a family of kernel measures
on $\mathbb{R}^{2}$ and let $\eta_{j}$ be the associated completely
positive maps on $L^{\infty}(\mathbb{R})$. 

Assume that each $K_{j}$ can be written as a sum of measures $K_{j}=\sum_{i\in S(j)}K_{j}^{(i)}$
with $K_{j}^{(i)}$ disjoint, and so that $K_{j}^{(i)}$ is supported
on the square $I_{j}^{(i)}\times I_{j}^{(i)}$ for a finite interval
$I_{j}^{(i)}$. Finally, suppose that there exist measures $\hat{K}_{j}^{(i)}$
on $I_{j}^{(i)}\times I_{j}^{(i)}$, so that $K_{j}^{(i)}$ is absolutely
continuous with respect to $\hat{K}_{j}^{(i)}$ and so that the associated
completely positive map
\[
\hat{\eta}_{j}^{(i)}:f\mapsto(x\mapsto\int f(y)d\hat{K}_{j}^{(i)}(x,y))
\]
defines a mixing $L^{\infty}(I_{j}^{(i)})$-bimodule.

Let $X_{j}$ be $\eta_{j}$-semicircular variables over $L^{\infty}(\mathbb{R})$,
and assume that $X_{j}$ are free with amalgamation over $L^{\infty}(\mathbb{R})$.
Then the semifinite von Neumann algebra $M=W^{*}(L^{\infty}(\mathbb{R}),X_{j}:j\in J)$
is solid, in the sense that if $A\subset M$ is any diffuse abelian
von Neumann subalgebra generated by its finite projections, then $A'\cap M$
is amenable.
\end{lem}

\begin{proof}
Denote by $\hat{X}_{j}^{(i)}$ the $\hat{\eta}_{j}^{(i)}$-semicircular
family, and assume that $\hat{X}_{j}^{(i)}$ are free with amalgamation
over $L^{\infty}(\mathbb{R})$. Let $\hat{M}=W^{*}(L^{\infty}(\mathbb{R}),\{\hat{X}_{j}^{(i)}:j\in J,i\in S(j)\})$.
Since $K_{j}=\sum_{i\in S(j)}K_{j}^{(i)}$ is a disjoint sum and $K_{j}^{(i)}$
is absolutely continuous with respect to $\hat{K}_{j}^{(i)}$, we conclude
that $M\subset\hat{M}$ and moreover $M$ is in the range of a conditional
expectation from $\hat{M}$. Thus is sufficient to prove that $\hat{M}$
is solid. 

By freeness with amalgamation, we know that $\hat{M}$ is the amalgamated
free product of the algebras $\hat{M}_{j}^{(i)}=W^{*}(L^{\infty}(\mathbb{R}),\hat{X}_{j}^{(i)})$.
Thus by \cite[Theorem 4.4]{Rigidity}, if $B\subset\hat{M}$ is an abelian algebra generated
by its finite projections and $B'\cap\hat{M}$ is non-amenable, then
$B\prec_{\hat{M}}\hat{M}_{j}^{(i)}$ for some $j\in J$ and $i\in S(j)$
and moreover it follows that $\hat{M}_{j}^{(i)}$ is not solid. But
\[
\hat{M}_{j}^{(i)}\cong L^{\infty}(\mathbb{R}\setminus I_{j}^{(i)})\oplus W^{*}(L^{\infty}(I_{J}^{(i)}),\hat{X}_{j}^{(i)})
\]
and the (finite) von Neumann algebra $W^{*}(L^{\infty}(I_{J}^{(i)}),\hat{X}_{j}^{(i)})$
is solid by \cite[Corollary 4.2]{Thin}, which is a contradiction.
\end{proof}
\begin{cor}
\label{cor:Solid}Suppose that $\pi$ is a mixing orthogonal
representation of $\mathbb{R}$ on a real Hilbert space $H_{\mathbb{R}}$,
and let $\alpha$ be the free Bogoljubov action on $\Phi(H_{\mathbb{R}})$
associated to $\pi$. Then the semi-finite von Neumann algebra $M=\Phi(H_{\mathbb{R}})\rtimes_{\alpha}\mathbb{R}$
is solid: if $B\subset M$ is an diffuse abelian subalgebra generated
by its finite projections, then $B'\cap M$ is amenable.
\end{cor}

\begin{proof}
Our goal is to apply Lemma \ref{lem:ACTrick}. Fix any decomposition
of $\pi$ into cyclic representations $(\pi_{j}:j\in J)$ with associated
cyclic vectors $\xi_{j}$ in such a way that the spectrum of $\pi_{j}(t)$
is contained in the set $\exp(iI_{j}t)$ for a finite subinterval
$I_{j}\subset\mathbb{R}$. Let us fix integers $n_{j}$ so that $I_{j}\subset[-n_{j},n_{j}]$.  Selecting a possibly different 
set of cyclic vectors and subrepresentations $\pi_j(t)$, we may assume that $\pi(t)=\bigoplus_j \pi_j(t)$, and
that the spectrum of the infinitesimal generator of $\pi_j$ is contained in $I_j$.   
Denote by $\mu_{j}$ the measures with Fourier transform
\[
\hat{\mu}_{j}(t)=\langle\xi_{j},\pi(t)\xi_{j}\rangle.
\]
By assumption that $\pi$ is mixing, $\lim_{t\to\pm\infty}\hat{\mu}_{j}(t)=0$.
Moreover, by construction, the support of $\mu_{j}$ is contained
in $[-n_{j},n_{j}]$.

Let $\eta_{j}:L^{\infty}(\mathbb{R})\to L^{\infty}(\mathbb{R})$ be
the completely positive map given by convolution with $\mu_{j}$.
Then $\eta_{j}$ has an associated kernel measure $K_{j}$ given by
$dK_{j}(x,y)=\mu_{j}(x-y)$.

Let $K_{j}^{(i)}$ denote the restriction of $K_{j}$ to the region
$[-4n_{j}+i,4n_{j}+i]\times[-4n_{j}+i,4n_{j}+i] \setminus [-4n_j + i-1, 4n_j + i-1] \times  [-4n_j + i-1, 4n_j + i -1 ] $, $i\in\mathbb{Z}$, and let $\hat{K}_{j}^{(i)}$ be the restriction of $K_j$ to $[-4n_{j}+i,4n_{j}+i]\times[-4n_{j}+i,4n_{j}+i]$. If we identify
$[-4n_{j}+i,4n_{j}+i]$ with the circle, then the completely positive
map associated to $\hat{K}_{j}^{(i)}$ is given by convolution with
the measure $\mu'_{j}$ whose Fourier transform is given by $k\mapsto\hat{\mu}_{j}(k/8n_{j})$;
since $\lim_{t\pm\infty}\hat{\mu}_{j}(t)=0$, it follows that $\lim_{k\to\pm\infty}\hat{\mu}'_{j}(k)=0$,
so that the hypothesis of Lemma \ref{lem:ACTrick} is satisfied.
\end{proof}

\subsection{Cocycle conjugacy.}

We are now ready to prove the main result of this paper:
\begin{thm}
\label{thm:SameSpectrum}Let $\pi_{1},\pi_{1}'$ be two mixing
orthogonal representations of $\mathbb{R}$, and assume that $\pi_{1}\otimes\pi_{1}\cong\pi_{1}$,
$\pi_{1}'\otimes\pi_{1}'\cong\pi_{1}'$. Denote by $\mathbf{1}$ the
trivial representation of $\mathbb{R}$. Let
\[
\pi=(\mathbf{1}\oplus\pi_{1})^{\oplus\infty},\qquad\pi'=(\mathbf{1}\oplus\pi_{1}')^{\oplus\infty},
\]
and let $\alpha$ (resp.,
$\alpha'$) be the corresponding free Bogoljubov actions of $\mathbb{R}$ on $L(\mathbb{F}_{\infty})$.
Then $\alpha$ and $\alpha'$ are cocycle conjugate iff $\mathscr{C}_{\pi_{1}}=\mathscr{C}_{\pi_{1}'}$. 
\end{thm}

\begin{proof}
If $\mathscr{C}_{\pi_{1}}=\mathscr{C}_{\pi_{1}'}$, then $\pi$ and
$\pi'$ are conjugate representations and the associated Bogoljubov
actions are conjugate; thus it is the opposite direction that needs
to be proved. The proof will be broken into several steps. 

If $H_{\pi}$ is the representation space of $\pi$, then $H_{\pi}\cong H_{0}\oplus H_{1}$
corresponding to the decomposition $\pi=\mathbf{1}^{\infty}\oplus\pi_{1}^{\oplus\infty}$.
Let $N=\Phi(H_{\pi})\cong L(\mathbb{F}_{\infty})$. Then $N$ decomposes
as a free product $N\cong\Phi(H_{0})*\Phi(H_{1})$; moreover, the
free Bogoljubov action is also a free product $\alpha=\mathbf{1}*\alpha_{1}$.
Note that the subalgebra $L(\mathbb{F}_{\infty})\cong\Phi(H_{0})\subset N$
is fixed by the action $\alpha$. In particular, the crossed product
$M=N\rtimes_{\alpha}\mathbb{R}$ decomposes as a free product:
\[
M=N\rtimes_{\alpha}\mathbb{R}\cong(\Phi(H_{0})\otimes L(\mathbb{R}))*_{L(\mathbb{R})}(\Phi(H_{1})\rtimes_{\alpha_{1}}\mathbb{R}).
\]

Let us assume that $\alpha$ and $\alpha'$ are cocycle conjugate.
Denote by $A=L_{\alpha}(\mathbb{R})\subset M$. Then $N\rtimes_{\alpha}\mathbb{R}\cong N\rtimes_{\alpha'}\mathbb{R}$
and thus (up to this identification, which we fix once and for all) also $L_{\alpha'}(\mathbb{R})\subset M$. Thus $A'\cap M\supset\Phi(H_{0})\cong L(\mathbb{F}_{\infty})$,
so that $A'\cap M$ is non-amenable. Exactly the same argument implies
that $L_{\alpha'}(\mathbb{R})\cap M$ is non-amenable. 

By \cite[Theorem 4.4]{Rigidity}, it follows from the amalgamated free product decomposition
of $M$ that $L_{\alpha'}(\mathbb{R})\prec_{M}\Phi(H_{0})\otimes L(\mathbb{R})$
or $L_{\alpha'}(\mathbb{R})\prec_{M}\Phi(H_{1})\rtimes_{\alpha_1}\mathbb{R}$.
But the latter is impossible by Corollary \ref{cor:Solid}, since
$\alpha_{1}$ comes from a mixing representation $\pi_{1}$.
Thus it must be that $L_{\alpha'}(\mathbb{R})\prec_{M}\text{\ensuremath{\Phi(H_{0})\otimes L(\mathbb{R})}}\cong L(\mathbb{F}_{\infty})\otimes L(\mathbb{R})$
and thus $L_{\alpha'}(\mathbb{R})\prec_{M}L_{\alpha}(\mathbb{R})$. 

By Theorem \ref{embedding}, $L_{\alpha'}(\mathbb{R})\prec_{M}L_{\alpha}(\mathbb{R})$
implies that there exists a nonzero partial isometry $v\in N$ such that
$v^{*}v\in N^{\alpha'}$, $vv^{*}\in N^{\alpha}$, and for all $x\in N$,
$\alpha_{t}(vxv^{*})=v\alpha'_{t}(x)v^{*}.$ 

Let $p=vv^{*}\in N^\alpha$, and denote by $\hat{\alpha}$
the restriction of $\alpha$ to $pNp$.

Let $H=H_{0}\oplus H_{1}$ be as above. By replacing $p\in\Phi(H_{0})$
with a subprojection and modifying $v$, we may assume that $\tau(p)=1/n$ for some $n$. Then we can find partial isometries
$v_{i}\in\Phi(H_{0})$, $i \in \{1,...,n\}$, such that $v_{i}v_{i}^{*}=p$
for all $i$ and $\sum_{i}v_{i}^{*}v_{i}=1.$ 

Let $\{s(h):h\in H_{1}\}$ be a semicircular family of generators
for $\Phi(H_{1})$. Then $N$ is generated by $\Phi(H_{0})\cup\{s(h):h\in H_{1}\}$,
so that $pNp$ is generated by $p\Phi(H_{0})p$ and $\{v_{i}s(h)v_{j}^{*}:1\leq i,j\leq n,h\in H_{1}\}$
\cite[Lemma 5.2.1]{FRV}. 

For $i,j\in \{1,...,n\}$ and $h\in H_{1}$, denote $S_{ij}(h)=\operatorname{Re}(n^{1/2}v_{i}s(h)v_{j}^{*})$,
$S'_{ij}(h)=\operatorname{Im}(n^{1/2}v_{i}s(h)v_{j}^{*})$. The
normalization is chosen so that in the compressed $W^{*}$-probability
space $(pNp,n\tau|_{pNp})$ these elements form a semicircular
family \cite[Prop. 5.1.7]{FRV}. So, all together, $pNp$ is generated
$*$-freely by $p\Phi(H_{0})p$ and the semicircular family $\{S_{ij}(h):h\in H_{1},1\leq i,j\leq n\}\cup\{S'_{ij}(h):h\in H_{1},1\leq i<j\leq n\}$.
The action of the restriction $\hat{\alpha}_{t}$ of $\alpha_{t}$
to $pNp$ is given, on these generators, as follows: $\hat{\alpha}_{t}(x)=x$
for $x\in p\Phi(H_{0})p$; $\hat{\alpha}_{t}(S_{ij}(h))=S_{ij}(\pi_t|_{H_{1}}(h))$,
$\hat{\alpha}_{t}(S'_{ij}(h))=S_{ij}'(\pi_t|_{H_{1}}(h))$. From this we see
that $\hat{\alpha}_{t}$ is once again a Bogoljubov automorphism but
corresponding to the representation $\mathbf{1}^{\oplus\infty}\oplus(\pi_{1}^{\oplus\infty})^{\oplus n^{2}}\cong\pi$.
Since by assumption $\pi_{1}\cong\pi_{1}\otimes\pi_{1}$, also $\pi\cong\pi\otimes\pi$
and so conjugacy of $\hat{\alpha}_{t}$ and $\hat{\alpha}'_{t}$ implies
equality of measure classes $\mathscr{C}_{\pi}$ and $\mathscr{C}_{\pi'}$
and thus of $\mathscr{C}_{\pi_{1}}$ and $\mathscr{C}_{\pi_{1}'}$. 
\end{proof}

%\bibliography{references}

\begin{bibdiv}
\begin{biblist}

\bib{Classification}{article}{
  title={Classification of a Family of Non-Almost Periodic Free Araki-Woods Factors},
  author={Houdayer, Cyril},
  author={Shlyakhtenko, Dimitri},
  author={Stefaan Vaes},
  journal={Journal of the European Mathematical Society},
  year={to appear}
  eprint={arXiv:1605.06057 [math.OA]}
  
}

\bib{BO}{book}{
  title={C*-algebras and Finite-dimensional Approximations},
  author={Brown, Nathanial Patrick},
  author = {Ozawa, Narutaka},
  volume={88},
  year={2008},
  publisher={American Mathematical Society}
}

\bib{A-valued}{article}{
  title={A-valued Semicircular Systems},
  author={Shlyakhtenko, Dimitri},
  journal={Journal of Functional Analysis},
  volume={166},
  number={1},
  pages={1--47},
  year={1999},
  publisher={Academic Press}
}

\bib{Thin}{article}{
  title={Thin {II$_1$} factors with no Cartan subalgebras},
  author={Krogager, Anna Sofie},
  author={Vaes, Stefaan},
  journal={Kyoto Journal of Mathematics},
  volume={59},
  number={4},
  pages={815-867},
  year={2019}
  eprint={arXiv:1611.02138 [math.OA]}
}

\bib{Rigidity}{article}{
  title={Rigidity of Free Product von Neumann Algebras},
  author={Houdayer, Cyril},
  author={Ueda, Yoshimichi},
  journal={Compositio Mathematica},
  volume={152},
  number={12},
  pages={2461--2492},
  year={2016},
  publisher={London Mathematical Society}
  eprint={arXiv:1507.02157 [math.OA]}
}

\bib{FRV}{book}{
  title={Free Random Variables},
  author={Voiculescu, Dan-Virgil},
  author={Dykema, Ken},
  author={Nica, Alexandru},
  number={1},
  year={1992},
  publisher={American Mathematical Society}
}

\end{biblist}
\end{bibdiv}
%\bibliographystyle{}

\end{document}